\documentclass[a4paper,12pt]{article}

\usepackage{pstricks}
\usepackage{graphicx,psfrag}
\usepackage{amsmath}
\usepackage{amssymb}
\usepackage{amsthm}
\usepackage{color}

      \newcommand {\del}  {\delta}          
              
                 \newcommand {\vphi} {\varphi}
      \newcommand {\lam}  {\lambda}         
                    
      \newcommand {\eee}   {e}       
                \newcommand {\Om}  {\Omega}
      \newcommand {\pl}   {\partial}        
      
      \newcommand {\kkk}   {k}           \newcommand {\KKK}   {K_0}
           \newcommand {\UUU}  {{\cal U}}
             
      \newcommand {\RRR}  {{\mathbb R}}     
              \newcommand {\MMM}  {{\cal M}}

      \newcommand {\FFF}  {{\cal F}}        
           
               \newcommand {\CCC}  {\mathfrak{C}}
        \newcommand {\fff}  {f}                                   \newcommand {\LLL}  {{\cal L}}
     \newcommand {\beq}  {\begin{equation}}
      \newcommand {\eeq}  {\end{equation}}

      \newtheorem{theorem}{Theorem}
      \newtheorem{lemma}{Lemma}
      
      \newtheorem{zam}{Remark}
      
            \newtheorem{corollary}{Corollary}

                    \newtheorem{conjecture}{Conjecture}

\title{On the structure of singular points of a solution\\ to Newton's least resistance problem}

\author{Alexander Plakhov\thanks{Center for R{\&}D in Mathematics and Applications, Department of Mathematics, University of Aveiro, Portugal and Institute for Information Transmission Problems, Moscow, Russia, plakhov@ua.pt}}

\begin{document}
\maketitle

          %\begin{flushright}  \it I dedicate this paper to the defenders of Ukraine  \end{flushright}

\begin{abstract}
We consider the following problem stated in 1993 by Buttazzo and Kawohl \cite{BK}: minimize the functional $\int\!\!\int_\Om (1 + |\nabla u(x,y)|^2)^{-1} dx\, dy$ in the class of concave functions $u: \Om \to [0,M]$, where $\Om \subset \RRR^2$ is a convex domain and $M > 0$. It generalizes the classical minimization problem, which was initially stated by I. Newton in 1687 in the more restricted class of radial functions. The problem is not solved until now; there is even nothing known about the structure of singular points of a solution.

In this paper we, first, solve a family of auxiliary 2D least resistance problems and, second, apply the obtained results to study singular points of a solution to our original problem. More precisely, we derive a necessary condition for a point being a ridge singular point of a solution and prove, in particular, that all ridge singular points with horizontal edge lie on the top level and zero level sets.
\end{abstract}

\begin{quote}
\end{quote}

\begin{quote}
\end{quote}

\section{Introduction}

Isaac Newton in his {\it Principia} \cite{N} considered the following minimization problem. A solid body moves with constant velocity in a sparse medium. Collisions of the medium particles with the body are perfectly elastic. The absolute temperature of the medium is zero, so as the particles are initially at rest. The medium is extremely rare, so that mutual interactions of the particles are neglected. As a result of body-particle collisions, the drag force acting on the body is created. This force is usually called {\it resistance}.

The problem is: given a certain class of bodies, find the body in this class with the smallest resistance. Newton considered the class of convex bodies that are rotationally symmetric and have fixed length along the direction of motion and fixed maximal width.

In modern terms the problem can be formulated as follows. Let a reference system $x_1,\, x_2,\, z$ be connected with the body and the $z$-axis coincide with the symmetry axis of the body. We assume that the particles move downward along the $z$-axis. Let the upper part of the body's surface be the graph of a concave radially symmetric function $z = u(x_1, x_2) = \vphi(\sqrt{x_1^2 + x_2^2})$, $x_1^2 + x_2^2 \le L^2$; then the resistance equals
$$
%F(\vphi) =
2\pi \rho v^2 \int_0^L \frac{1}{1 + \vphi'(r)^2}\, r\, dr,
$$
where the density of the medium $\rho$ and the scalar velocity of the body $v$ are assumed to be constant. The problem is to minimize %$F(\vphi)$
 the resistance in the class of convex monotone decreasing functions $\vphi : [0,\, L] \to \RRR$ satisfying $0 \le \vphi \le M$. Here $M$ and $L$ are the parameters of the problem: $M$ is length of the body and $2L$ is its maximal width.

Newton gave a geometric description of the solution to the problem and did not explain how the solution was obtained. The optimal function bounds a convex body that looks like a truncated cone with slightly inflated lateral boundary. An optimal body, corresponding to the case when the length is equal to the maximal width, is shown in Fig.~\ref{figNewton}.
    \begin{figure}[h]
\centering
\hspace*{6mm}
%\rotatedown{
\includegraphics[scale=0.45]{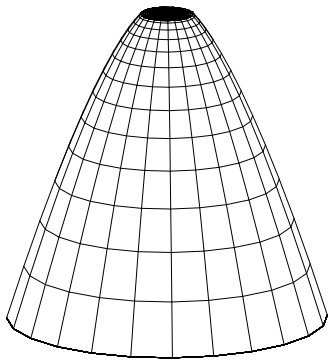} %}
\caption{A solution to the rotationally symmetric Newton problem.}
\label{figNewton}
\end{figure}

Starting from the pioneering paper by Buttazzo and Kawohl \cite{BK}, the problem of minimal resistance has been studied in various classes of (generally) nonsymmetric and/or (generally) nonconvex bodies; see, e.g., \cite{CL1,CL2,BrFK,BFK,BG97,LO,W,AP,SIREV,ARMA,OMT,bookP}.

In this paper we consider the generalization of the original Newton's problem to the class of convex bodies without the assumption of axial symmetry. The problem is as follows:

\begin{quote}
Minimize
\beq\label{Resist}
F(u) = \int\!\!\!\int_\Om \frac{1}{1 + |\nabla u(x_1,x_2)|^2}\, dx_1 dx_2
\eeq
in the class of functions
$$
\CCC_M = \{ u : \Om \to \RRR :\, 0 \le u \le M,\ u  \text{ is concave} \}.
$$
\end{quote}

Here $\Om \subset \RRR^2$ is a compact convex set with nonempty interior int$(\Om)$, and $M > 0$ is the parameter of the problem.

Surprisingly enough, this problem is still poorly understood. It is known that there exists at least one solution \cite{M,BFK}. Let $u$ be a solution; then $u\rfloor_{\pl\Om} = 0$ \cite{boundary} and at any regular point $x = (x_1, x_2)$ of $u$ we have either $|\nabla u(x)| \ge 1$, or $|\nabla u(x)| = 0$ \cite{BFK}. Moreover, if the set $L = \{ x : u(x) = M \} \subset \RRR^2$ has nonempty interior then we have $\lim_{\stackrel{x \to \bar x}{x \not\in L}} |\nabla u(x)| = 1$ for almost all $\bar x \in \pl L$ \cite{MMS}. If $u$ is regular in an open set $\UUU \subset \Om$, then the surface graph$\big( u\rfloor_\UUU \big) = \{ (x, u(x)) : x \in \UUU \}$ does not contain extreme points of the convex body\footnote{A convex body is a convex compact set with nonempty interior.}
$$
C_u = \{ (x,z) :\, x \in \Om,\ 0 \le z \le u(x) \},
$$
and therefore, is developable \cite{nose stretching}.

This paper is devoted to studying singular points of an optimal function $u$ or, equivalently, singular points of the corresponding convex body $C = C_u$.

A singular point $r_0$ on the boundary of a convex body $C$ is called {\it conical point}, if the tangent cone to $C$ at $r_0$ is not degenerate, and {\it ridge point}, if the tangent cone degenerates into a dihedral angle (see, e.g., \cite{Pogorelov}). In this paper we consider ridge points, postponing the study of conical points to the future.

Let $r_0 = (x, u(x))$,\, $x \in \text{int}(\Om)$ be a ridge point of $C = C_u$ and $e_1$ and $e_2$ be the outward normals to the faces of the corresponding dihedral angle. Introduce some additional notation. Let $l$ be the edge of this angle, and denote by $\theta \in [0,\, \pi/2)$ the angle between $l$ and the $x$-plane. Draw a plane orthogonal to $l$, that is, parallel to $e_1$ and $e_2$. The angle between the plane and the $z$-axis is $\theta$.

Take an orthonormal basis $\fff_1,\, \fff_2$ in this plane, so as $\fff_1$ is horizontal and the $z$-coordinate of $\fff_2$ is positive. Let $e_1$ and $e_2$ form the angles $\vphi_1$ and $\vphi_2$ with $\fff_2$ counted in a certain direction, $-\pi/2 < \vphi_2 < \vphi_1 < \pi/2$. In appropriate coordinates $x_1,\, x_2$ on the $x$-plane, each vector $e_i,\, i = 1,\, 2$ takes the form
$$
e_i = (-\sin\vphi_i,\, \cos\vphi_i \sin\theta,\, \cos\vphi_i \cos\theta).
$$

The triple of angles $\theta$, $\vphi_1$, and $\vphi_2$ uniquely defines a dihedral angle, up to motions of the $x$-plane.

The following Theorem \ref{t0} is the main result of this paper. It establishes a necessary condition for a ridge singular point of an optimal body.

 \begin{theorem}\label{t0}
Let $\theta$, $\vphi_1$, and $\vphi_2$ be the angles associated with a ridge point of an optimal body.

(a) If $\theta \in [\pi/4,\, \pi/2)$ then $\vphi_1$ and $\vphi_2$ satisfy the inequalities
$$
2\vphi_1 + \vphi_2 \le \pi/2, \qquad \vphi_1 + 2\vphi_2 \ge -\pi/2;
$$
the corresponding set of admissible points $(\vphi_1, \vphi_2)$ is the triangle (i) in Fig.~\ref{fig2Dmeas}.

(b) If $\theta \in (0,\, \pi/4)$ then
$$
2\vphi_1 + \vphi_2 \le \pi/2, \quad \vphi_1 + 2\vphi_2 \ge -\pi/2, \quad |\vphi_1| \ge \vphi_*, \quad |\vphi_2| \ge \vphi_*,
$$
where $\vphi_* =\arccos\big( \frac{1/\sqrt 2}{\cos\theta} \big)$. The corresponding set is shown lightgray in Fig.~\ref{fig2Dab}\,(a).

(c) If $\theta = 0$ then either $(\vphi_1, \vphi_2) = (\pi/4, 0)$, or $(\vphi_1, \vphi_2) = (0, -\pi/4)$, or $\vphi_1$ and $\vphi_2$ satisfy
$$
2\vphi_1 + \vphi_2 \le \pi/2, \quad \vphi_1 + 2\vphi_2 \ge -\pi/2, \quad \vphi_1 \ge \pi/4, \quad \vphi_2 \le -\pi/4.
$$
The corresponding set is the union of two points and a quadrangle shown lightgray in Fig.~\ref{fig2Dab}\,(b).
\end{theorem}

\begin{zam}\label{z1}
Observe that almost all points on the boundary $\pl\LLL$ of the upper level set (ULS) $\LLL = \{ (x, u(x)) : u(x) = M \}$ are ridge points with horizontal edge. The rest of the points on $\pl\LLL$ are (finitely or countably many) conical points.

If $\LLL$ has nonempty interior, each ridge point on $\pl\LLL$ corresponds to one of the points $(\pi/4, 0)$ and $(0, -\pi/4)$; see Fig.~\ref{fig2Dab}\,(b). If, otherwise, $\LLL$ is a line segment, each ridge point on $\pl\LLL$ corresponds to a point of the quadrangle shown lightgray in Fig.~\ref{fig2Dab}\,(b). There are no other ridge points with horizontal edge.
\end{zam}

Let us formulate this observation as a corollary of the theorem.

  \begin{corollary}\label{cor2}
All ridge points with horizontal edge of an optimal body lie on the boundary of the upper level set $C_u \cap \{ z = M \}$ and on the boundary of the zero level set $C_u \cap \{ z = 0 \}$.
\end{corollary}

         \begin{figure}[h]
\centering
\hspace*{6mm}
%\rotatedown{
\includegraphics[scale=0.2]{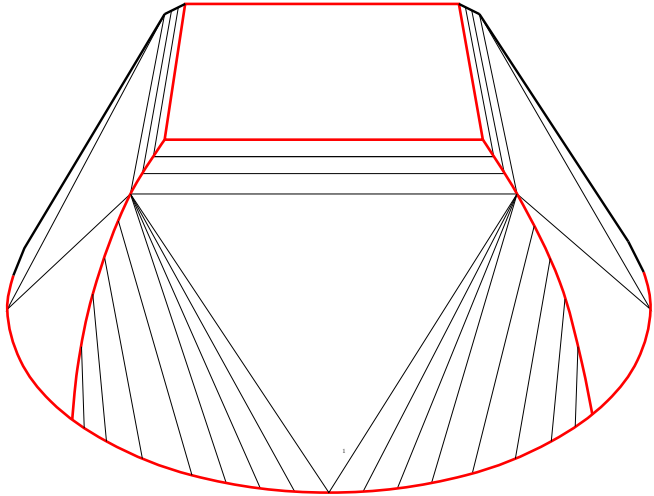} %}
\caption{An optimal body with the upper level set being a square.}
\label{fig square body}
\end{figure}

\begin{zam}\label{z2}
Numerical simulation \cite{W} indicates that if $\Om$ is a unit circle and $M \lesssim 1.5$ then the ULS of an optimal body is a regular polygon and the lateral surface of the body is foliated by line segments and planar triangles. Moreover, a part of the surface near ULS is foliated by horizontal segments. See Fig.~\ref{fig square body}, where the ULS is a square. It follows from Corollary \ref{cor2} that all points of this part of surface, except for endpoints of the segments, are regular.

Further, there are several singular curves on the lateral surface joining the endpoints of the polygon with points of the base $\pl\Om \times \{ 0 \}$. Each point of such a curve is a ridge point and satisfies $\theta \ne 0,\, \vphi_1 = -\vphi_2$. Thus, we come to the following conjecture.

%\beqin{quotation}
\begin{conjecture}
If $\Om$ is a unit circle and the ULS has nonempty interior then all singular points outside ULS are ridge points and satisfy $\theta \ne 0,\, \vphi_1 = -\vphi_2$.
\end{conjecture}
%\end{quotation}

On the other hand, if the ULS is a line segment or $\Om$ is not a circle, the structure of singular points on the lateral surface is unclear.
\end{zam}

\section{Surface area measure of convex bodies}

Here we provide some information concerning the surface area measure of convex bodies and representation of the resistance in terms of surface area measure, which will be needed later on.

Let $C$ be a convex body in $\RRR^d$. Denote by $n_r$ the outward normal to $C$ at a regular point $r \in \pl C$. The surface measure of $C$ is the Borel measure $\nu_C$ in $S^{d-1}$ defined by
$$
\nu_C(A) = \text{Leb}(\{ r \in \pl C : n_r \in A \})
$$
for any Borel set $A \subset S^2$, where Leb means the $(d-1)$-dimensional Lebesgue measure on $\pl C$. We will only need the cases $d=2$ and 3.

It is well known that the surface area measure satisfies the equation
$$
\int_{S^{d-1}} n\, d\nu_C(n) = \vec 0.
$$

In a similar way one defines the measure induced by a Borel subset of $\pl C$.

The functional $F(u)$ in \eqref{Resist} can be represented in terms of surface area measure as $F(u) = \FFF(C_u)$, where
$$
\FFF(C) = \int_{S^2} (n_3)_+^3\, d\nu_C(n).
$$
Here $n = (n_1, n_2, n_3)$ and $z_+ = \max \{ z, 0 \}$ means the positive part of a real number $z$.

A similar formula holds in the 2D case. Let $u : [a,\, b] \to \RRR$ be a concave function and $C_u$ be the convex body bounded above by the graph of $u$ and below by the segment joining the points $(a, u(a))$ and $(b, u(b))$. Then we have $F(u) = \FFF(C_u)$, where
$$
F(u) = \int_a^b \frac{1}{1 + u'^2(x)}\, dx
$$
and
$$
\FFF(C) = \int_{S^1} (n_3)_+^3\, d\nu_C(n),
$$
with $n = (n_1, n_3)$.

Let $\mu_C$ be the push-forward measure of $\nu_C$ under the map from $S^1$ to $(-\pi,\, \pi]$ given by $(-\sin\vphi, \cos\vphi) \mapsto \vphi$. Then the latter formula can be rewritten as
$$
\FFF(C) = \int_{-\pi/2}^{\pi/2} (\cos\vphi)^3\, d\mu_C(\vphi).
$$

\section{2D problems of minimal resistance}

Here we state and solve some auxiliary 2-dimensional problems of minimal resistance. They will be used in the next section.

The direct generalization of Newton's problem to the 2D case is as follows: given $M > 0$, minimize the integral $\int_{0}^1 (1 + u'^2(x))^{-1} dx$ in the class of concave functions $u : [0,\, 1] \to \RRR$ such that $u(0) = M$, $u(1) = 0$ and $u'(x) \le 0$. This problem and its solution were stated in \cite{BK}; the solution is
$$
u(x) =
\left\{
\begin{array}{ll}
\min \{ M,\, 1-x \}, & \text{if } \ M < 1, \\
 M(1 - x), & \text{if } \ M \ge 1.
 \end{array}
 \right.
$$

We will consider the problem in a slightly different form. Suppose that we are given 4 real numbers $x_0 > 0,\, z_0,\, \kkk_1,\, \kkk_2$ such that $\kkk_2 < z_0/x_0 < \kkk_1$ and $x_0^2 + z_0^2 = 1$, and denote $\KKK = z_0/x_0$.

\begin{quote}
{\bf Problem.} \ Minimize the integral
\beq\label{pro2D}
\int_0^{x_0} \frac{1}{1 + u'^2(x)}\, dx
\eeq
in the class $\CCC_{\KKK,\kkk_1,\kkk_2}$ of concave functions $u : [0,\, x_0] \to \RRR$ such that $u(0) = 0$, $u(x_0) = z_0$, and $\kkk_2 \le u'(x) \le \kkk_1$.
\end{quote}

The problem can be interpreted as minimizing the resistance in the class of planar convex bodies that are contained in the triangle $ABC$ and contain the points $A = (0,0)$ and $B = (x_0, z_0)$. The slopes of the sides $AC$ and $BC$ are $\kkk_1$ and $\kkk_2$, respectively; see Fig.~\ref{fig triang}. It is assumed that there is a flow incident on the body moving downward in the $z$-direction.

                                       \begin{figure}[h]
\centering
%\hspace*{6mm}
\includegraphics[scale=0.15]{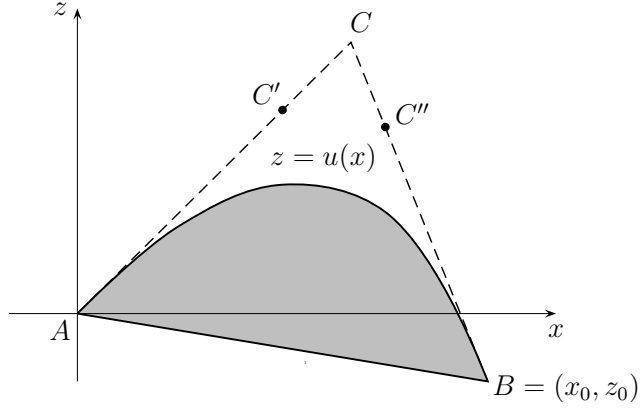}
\caption{The triangle $ABC$ and a typical convex body from the chosen class.}
\label{fig triang}
\end{figure}

\begin{lemma}\label{l 2D}
The solution $u$ to problem \eqref{pro2D} is unique.

(i) Let $-\sqrt{1 + \kkk_1^2} \le \kkk_1 + \kkk_2 \le \sqrt{1 + \kkk_2^2}$. Then the derivative $u'(x)$ takes two values, $\kkk_1$ and $\kkk_2$, and therefore, $u(x) = \min \{ \kkk_1 x,\, z_0 + \kkk_2(x - x_0) \}$, and the graph of $u$ is the broken line $ACB$.

(ii) Let $\kkk_1 \le -1/\sqrt 3$ or $\kkk_2 \ge 1/\sqrt 3$. Then $u(x) = \KKK x$, and therefore, the graph of $u$ is the segment $AB$.

(iii) Let $\kkk_2 < 1/\sqrt 3$,\, $\kkk_1 > \sqrt{1 + \kkk_2^2} - \kkk_2$. Then

\hspace*{4mm}(a) if $\kkk_2 < \KKK < \sqrt{1 + \kkk_2^2} - \kkk_2$ then $u'(x)$ takes two values, $\kkk_2$ and $\sqrt{1 + \kkk_2^2} - \kkk_2$,
and the graph of $u$ is the broken line $AC''B$, where $C''$ is a point on $CB$;

\hspace*{4mm}(b) if $\sqrt{1 + \kkk_2^2} - \kkk_2 \le \KKK < \kkk_1$ then $u(x) = \KKK x$.

(iv) Let $\kkk_1 > -1/\sqrt 3$,\, $\kkk_2 < -\sqrt{1 + \kkk_1^2} - \kkk_1$. Then

\hspace*{4mm}(a) if $-\sqrt{1 + \kkk_1^2} - \kkk_1 < \KKK < \kkk_1$ then $u'(x)$ takes the values $\kkk_1$ and $-\sqrt{1 + \kkk_1^2} - \kkk_1$,
and the graph of $u$ is the broken line $AC'B$, where $C'$ is a point on $AC$;

\hspace*{4mm}(b) if $\kkk_2 < \KKK \le -\sqrt{1 + \kkk_1^2} - \kkk_1$ then $u(x) = \KKK x$.
\end{lemma}

\begin{proof}
Denote by $p(\xi)$ the restriction of the function $1/(1 + \xi^2)$ on the segment $[\kkk_2,\, \kkk_1]$, and by $\bar p(\xi)$, the maximal convex function on $[\kkk_2,\, \kkk_1]$ satisfying $\bar p(\xi) \le p(\xi)$. In other words, the epigraph of $\bar p$ is the convex hull of the epigraph of $p$.

One easily checks that the function $1/(1 + \xi^2)$ is convex on the intervals $(-\infty,\, -1/\sqrt 3]$ and $[1/\sqrt 3,\, +\infty)$ and concave on the interval $[-1/\sqrt 3,\, 1/\sqrt 3]$. Note that the graph of $\bar p$ cannot contain two strictly convex arcs separated by a line segment; otherwise the line containing this segment is tangent to the graph of $z = 1/(1 + \xi^2)$ at two points, which is impossible.

Likewise, the graph of $\bar p$ does not contain two line segments separated by a (strictly convex) arc of graph$(p)$. Indeed, otherwise the endpoints of each segment bound a part of graph$(p)$ containing at least one point where $p'' < 0$. This means that there exist points $\xi_1 < \xi_2 < \xi_3$ such that $p''(\xi_1) < 0,\, p''(\xi_2) > 0,\, p''(\xi_3) < 0$, which is impossible.

Therefore there are 4 possibilities for graph$(\bar p)$:
\vspace{1mm}

(i) it is the line segment joining the endpoints of graph$(p)$;
\vspace{1mm}

(ii) it coincides with graph$(p)$;
\vspace{1mm}

(iii) it is the union of a line segment on the left and a part of graph$(p)$ on the right;
\vspace{1mm}

(iv) vice versa: it is the union of a part of graph$(p)$ on the left and a line segment on the right.
\vspace{1mm}

The case (ii) is realized, if and only if the intervals $[\kkk_2,\, \kkk_1]$ and $(-1/\sqrt 3,\, 1/\sqrt 3)$ do not intersect, that is, either $\kkk_2 < \kkk_1 \le -1/\sqrt 3$ or $1/\sqrt 3 \le \kkk_2 < \kkk_1$. Let us consider conditions for realization of cases (iii) and (iv).

The necessary and sufficient conditions for graph$(\bar p)$ to be the union of a line segment projecting to $[\kkk_2,\, \bar\kkk]$ and a part of graph$(p)$ projecting to $[\bar\kkk,\, \kkk_1]$ are that $\kkk_2 < 1/\sqrt 3$,\, $\bar\kkk > 1/\sqrt 3$, and the straight line through $(\kkk_2, p(\kkk_2))$ and $(\bar\kkk, p(\bar\kkk))$ is tangent to graph$(p)$ at the point $(\bar\kkk, p(\bar\kkk))$. The condition of tangency looks as follows:
$$
\frac{p(\bar\kkk) - p(\kkk_2)}{\bar\kkk - \kkk_2} = p'(\bar\kkk).
$$
Using that $p(\xi) = 1/(1 + \xi^2)$, after a simple algebra one obtains the equation $\bar\kkk^2 + 2\kkk_2\bar\kkk - 1 = 0$, which has two solutions. The solution $\bar\kkk = -\kkk_2 - \sqrt{1 + \kkk_2^2} < 0 < 1/\sqrt 3$ does not serve, therefore we have $\bar\kkk = -\kkk_2 + \sqrt{1 + \kkk_2^2}$. One easily sees that $\bar\kkk > 1/\sqrt 3$.

Thus, the conditions (iii) for graph$(\bar p)$ be composed of a line segment on the left and a strictly convex part on the right are:
$$
\kkk_2 < 1/\sqrt 3, \qquad \kkk_1 > -\kkk_2 + \sqrt{1 + \kkk_2^2}.
$$
In a similar way one obtains that the conditions (iv) for graph$(\bar p)$ to be the union of a a strictly convex part on the left and line segment on the right are:
$$
\kkk_1 > -1/\sqrt 3, \qquad \kkk_2 < -\kkk_1 - \sqrt{1 + \kkk_1^2}.
$$
The resting part corresponding to case (i) is described by the inequalities
$$
-\sqrt{1 + \kkk_1^2} \le \kkk_1 + \kkk_2 \le \sqrt{1 + \kkk_2^2};
$$
see Fig.~\ref{fig2Du}.

                                       \begin{figure}[h]
\centering
%\hspace*{6mm}
\includegraphics[scale=0.15]{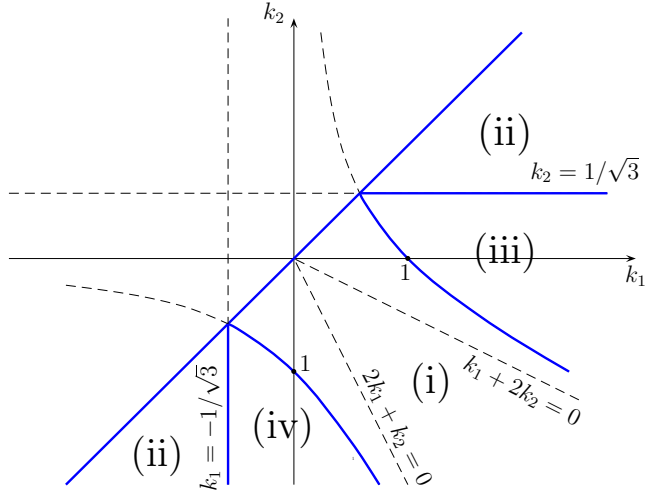}
\caption{The domains (i), (ii), (iii), (iv).}
\label{fig2Du}
\end{figure}

In the case (i) the value $\int_0^{x_0} \bar p(u'(x))\, dx$ is constant for all $u \in \CCC_{\KKK,\kkk_1,\kkk_2}$. We have the inequality $\bar p \le p$, with the equality being attained only at $\kkk_1$ and $\kkk_2$. Therefore,
$$
\int_{0}^{x_0} \bar p(u'(x))\, dx \le \int_{0}^{x_0} p(u'(x))\, dx,
$$
and the equality is attained {\it iff} $u'$ takes only the values $\kkk_1$ and $\kkk_2$. Thus, claim (i) of the lemma is proved.

In the case (ii) the function $p$ is strictly convex, and by Jensen's inequality,
$$
\int_{0}^{x_0} p(u'(x))\, dx \ge x_0\, p(\KKK),
$$
with the equality being attained {\it iff} $u'(x) = \KKK$ for all $x$. Thus, claim (ii) is also proved.

In the case (iii) the function $\bar p$ is linear on $[\kkk_2,\, \bar \kkk]$ and strictly convex on $[\bar \kkk,\, \kkk_1]$, with $\bar\kkk = -\kkk_2 + \sqrt{1 + \kkk_2^2}$. Besides, $\bar p = p$ on $\{ \kkk_2 \} \cup [\bar\kkk,\, \kkk_1]$, and $\bar p < p$ on $(\kkk_2,\, \bar\kkk)$. We have
$$
\int_{0}^{x_0} p(u'(x))\, dx \ge \int_{0}^{x_0} \bar p(u'(x))\, dx \ge x_0\, \bar p(\KKK).
$$
The former inequality becomes equality {\it iff} $u'(x) \in \{ \kkk_2 \} \cup [\bar\kkk,\, \kkk_1]$ for all $x$. In the case (iii)(a), $\kkk_2 < \KKK < \bar\kkk$, the latter inequality becomes equality {\it iff} $u' \in [\kkk_2,\, \bar \kkk]$. In the case (iii)(b), $\bar\kkk \le \KKK < \kkk_1$, the latter inequality becomes equality {\it iff} $u' = \KKK$. We conclude that the integral $\int_{0}^{x_0} p(u'(x))\, dx$ takes its minimal value $x_0\, \bar p(\KKK)$, if $u'$ takes the values $\kkk_2$ and $\bar\kkk$ in the case (iii)(a), and if $u' = \KKK$ in the case (iii)(b). Claim (iii) is proved.

The proof of claim (iv) is completely analogous to the proof of claim (iii). Lemma \ref{l 2D} is proved.
\end{proof}

One can reformulate these results in terms of the surface measure.

Denote $\vphi_0 = \arctan\KKK$, $\vphi_1 = \arctan\kkk_1$, $\vphi_2 = \arctan\kkk_2$. We have $\vphi_2 < \vphi_0 < \vphi_1$. We shall use the notation $\eee_\vphi = (-\sin\vphi, \cos\vphi) \in S^1$. The 2-dimensional convex body bounded below by the segment $AB$ and above by the graph of $u$ induces the measure $-\del_{\vphi_0} + \mu_u$, where the measure $\mu = \mu_u$ is supported on $[\vphi_2,\, \vphi_1]$ and satisfies the relation
\beq\label{mu}
\int_{\vphi_2}^{\vphi_1} e_\vphi\, d\mu(\vphi) = e_{\vphi_0}.
\eeq

Denote by $\MMM_{\vphi_0,\vphi_1,\vphi_2}$ the set of measures $\mu$ on $[\vphi_2,\, \vphi_1]$ satisfying \eqref{mu}. It is well known that there is a one-to-one correspondence between the set of measures $\MMM_{\vphi_1,\vphi_2,\vphi_0}$ and the set of functions $\CCC_{\KKK,\kkk_1,\kkk_2}$.

Now problem \eqref{pro2D} can be reformulated as follows.

\begin{quote}
{\bf Problem.} \ Given $-\pi/2 \le \vphi_2 < \vphi_0 < \vphi_1 \le \pi/2$, minimize the integral
\beq\label{pro2Dmeas}
\Phi(\mu) = \int_{\vphi_2}^{\vphi_1} (\cos\vphi)^3\, d\mu(\vphi)
\eeq
in the class of measures $\MMM_{\vphi_0,\vphi_1,\vphi_2}$.
\end{quote}

Choose $\lam_1 > 0$ and $\lam_2 > 0$ so as $\lam_1 \eee_{\vphi_1} + \lam_2 \eee_{\vphi_2} = \eee_{\vphi_0}$.

Define the values $\bar\lam_1$ and $\bar\lam_2$ in the two particular cases, which are indicated as (iii)(a) and (iv)(a) in the following Lemma \ref{l 2D meas} and cannot occur simultaneously.

If $\vphi_2 <  \pi/6$,\, $2\vphi_1 + \vphi_2 > \pi/2$, and $\vphi_2 < \vphi_0 < (\pi-2\vphi_2)/4$ (case (iii)(a)\,), choose $\bar\lam_1 > 0$ and $\bar\lam_2 > 0$ so as
$$\bar\lam_1 \eee_{(\pi-2\vphi_2)/4} + \bar\lam_2 \eee_{\vphi_2} = \eee_{\vphi_0}.$$

If $\vphi_1 > -\pi/6$,\, $\vphi_1 + 2\vphi_2 < -\pi/2$, and $-(\pi + 2\vphi_1)/4 < \vphi_0 < \vphi_1$ (case (iv)(a)\,), choose $\bar\lam_1 > 0$ and $\bar\lam_2 > 0$ so as
$$\bar\lam_1 \eee_{\vphi_1} + \bar\lam_2 \eee_{-(\pi + 2\vphi_1)/4} = \eee_{\vphi_0}.$$

The following Lemma \ref{l 2D meas} is just a reformulation of Lemma \ref{l 2D} in terms of the angles $\vphi_0$, $\vphi_1$, $\vphi_2$.

\begin{lemma}\label{l 2D meas}
The solution $\mu$ to problem \eqref{pro2Dmeas} is unique.

(i) Let $2\vphi_1 + \vphi_2 \le \pi/2$ and $\vphi_1 + 2\vphi_2 \ge -\pi/2$; then $\mu = \lam_1 \del_{\eee_{\vphi_1}} + \lam_2 \del_{\eee_{\vphi_2}}$.

(ii) Let $\vphi_1 \le -\pi/6$ or $\vphi_2 \ge \pi/6$; then $\mu = \del_{\eee_{\vphi_0}}$.

(iii) Let $\vphi_2 <  \pi/6$ and $2\vphi_1 + \vphi_2 > \pi/2$; then

\hspace*{4mm}(a) if $\vphi_2 < \vphi_0 < (\pi-2\vphi_2)/4$ then $\mu = \bar\lam_1 \del_{\eee_{(\pi-2\vphi_2)/4}} + \bar\lam_2 \del_{\eee_{\vphi_2}}$.

\hspace*{4mm}(b) if $(\pi-2\vphi_2)/4 \le \vphi_0 < \vphi_1$ then $\mu = \del_{\eee_{\vphi_0}}$.

(iv) Let $\vphi_1 > -\pi/6$, $\vphi_1 + 2\vphi_2 < -\pi/2$; then

\hspace*{4mm}(a) if $-(\pi + 2\vphi_1)/4 < \vphi_0 < \vphi_1$ then $\mu = \bar\lam_1 \del_{\eee_{\vphi_1}} + \bar\lam_2 \del_{\eee_{-(\pi + 2\vphi_1)/4}}$;

\hspace*{4mm}(b) if $\vphi_2 < \vphi_0 \le -(\pi + 2\vphi_1)/4$ then $\mu = \del_{\eee_{\vphi_0}}$.

See Fig.~\ref{fig2Dmeas}.
\end{lemma}

                                       \begin{figure}[h]
\centering
%\hspace*{6mm}
\includegraphics[scale=0.15]{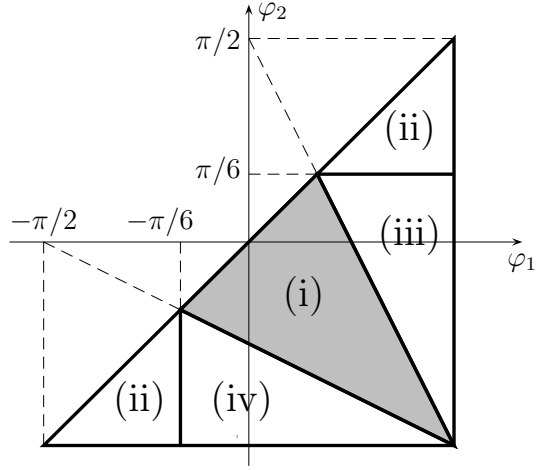}
\caption{The domains (i), (ii), (iii), (iv) in terms of $\vphi_1$ and $\vphi_2$.}
\label{fig2Dmeas}
\end{figure}

\section{Proof of Theorem \ref{t0}}
%\vspace{2mm}

Recall that $r_0 = (x, u(x)),\, x \in \text{int}(\Om)$ is a ridge point of $C = C_u$ and $e_1$ and $e_2$ are the outward normals to the faces of the corresponding dihedral angle. Choose a vector $e$ on the smaller arc of the great circle in $S^2$ through $e_1$ and $e_2$; we have $|e| = 1$ and $e = \lam_1 e_1 + \lam_2 e_2$ for some $\lam_1 > 0$, $\lam_2 > 0$.

Let $\theta$ be the angle between a plane parallel to $e_1$ and $e_2$ and the $z$-axis. In appropriate coordinates $x_1,\, x_2$ on the $x$-plane the vectors $e_1,\, e_2,$ and $e$ have the form
$$
e_i = (-\sin\vphi_i,\, \cos\vphi_i \sin\theta,\, \cos\vphi_i \cos\theta), \quad i = 1,\ 2, \quad \vphi_2 < \vphi_1;
$$
$$
e = (-\sin\vphi_0,\, \cos\vphi_0 \sin\theta,\, \cos\vphi_0 \cos\theta) \quad \text{for some} \quad \vphi_2 < \vphi_0 < \vphi_1.
$$

Take $t > 0$ and consider the convex body
$$
C_t = C \cap \{ r : \langle r - r_0,\, e \rangle \ge -t \};
$$
it is the piece of $C$ cut off by the plane with the normal vector $e$ at the distance $t$ from $r_0$. Here $\langle \cdot \,, \cdot \rangle$ means the scalar product. The body $C_t$ is bounded by the planar domain
$$
B_t = C \cap \{ r : \langle r - r_0,\, e \rangle = -t \}
$$
and the convex surface
\beq\label{St}
S_t = \pl C \cap \{ r : \langle r - r_0,\, e \rangle \ge -t \};
\eeq
that is, $\pl C_t = B_t \cup S_t$.

Let $\nu_{S_t}$ be the measure induced by the surface $S_t$. It is proved in Theorem 2 in \cite{MMS} that there exists at least one weak partial limit of the normalized measure $\frac{1}{|B_{t}|} \nu_{S_t}$ as $t \to 0$. In other words, there exists at least one sequence $t_i \to 0^+$ as $i \to \infty$ such that $\frac{1}{|B_{t_i}|} \nu_{S_{t_i}}$ weakly converges to a measure $\nu_*$ as $i \to \infty$. Moreover, the support of the limiting measure $\nu_*$ is contained in the smaller arc of the great circle through the vectors $e_1$ and $e_2$ and contains these vectors. Additionally, one has
$$
 \int_{S^2} n\, d\nu_*(n) = e.
$$

Since the body $C$ is optimal, the following inequality holds
\beq\label{ineq}
\lim_{i\to\infty} \frac{1}{|B_{t_i}|} \big[ \FFF(C) - \FFF(C_{t_i}) \big] = \int_{S^2} (n_3)^3 d\nu_*(n) - (e_3)^3 \le 0.
\eeq

Let $\mu_*$ be the push-forward measure of $\nu_*$ under the map from the great circle through $e_1$ and $e_2$ to $(-\pi,\, \pi]$ defined by $(-\sin\vphi,\, \cos\vphi \sin\theta,\, \cos\vphi \cos\theta) \mapsto \vphi$. The support of $\mu_*$ is contained in $[\vphi_2,\, \vphi_1]$ and contains the points $\vphi_2$ and $\vphi_1$, therefore $\mu_*$ is not an atom of the form $\del_\vphi$,\, $\vphi \in (\vphi_2,\, \vphi_1)$. Additionally, one has
$$
\int_{\vphi_2}^{\vphi_1} (-\sin\vphi,\, \cos\vphi \sin\theta,\, \cos\vphi \cos\theta)\, d\mu_*(\vphi) = (-\sin\vphi_0,\, \cos\vphi_0 \sin\theta,\, \cos\vphi_0 \cos\theta).
$$
It follows that $\mu_* \in \MMM_{\vphi_0,\vphi_1,\vphi_2}$.

Using that $n_3 = \cos\vphi \cos\theta$ and $e_3 = \cos\vphi_0 \cos\theta$, the inequality in \eqref{ineq} can be rewritten in terms of $\mu_*$ as
$$
\int_{\vphi_2}^{\vphi_1} (\cos\vphi \cos\theta)^3 d\mu_*(\vphi) \le (\cos\vphi_0 \cos\theta)^3.
$$
Using the notation
$$\Phi(\mu) = \int_{\vphi_2}^{\vphi_1} (\cos\vphi)^3 d\mu(\vphi),$$
we obtain the inequality
\beq\label{i}
\Phi(\mu_*) \le \Phi(\del_{\vphi_0}). % \quad \text{for all} \ \ \vphi_0 \in (\vphi_2,\, \vphi_1).
\eeq

By Lemma \ref{l 2D meas}, in the cases (ii), (iii), and (iv) one can choose $\vphi_0 \in (\vphi_2,\, \vphi_1)$ so as the unique minimum of $\Phi$ is attained at $\del_{\vphi_0}$, and so, $\Phi(\mu_*) > \Phi(\del_{\vphi_0})$, in contradiction with \eqref{i}. Thus, only the case (i) can be realized.

Consider three possible cases for $\theta$.
\vspace{1mm}

(a) $\theta \in [\pi/4,\, \pi/2)$. Since the case (i) in Lemma \ref{l 2D meas} is realized, we have
$$2\vphi_1 + \vphi_2 \le \pi/2 \quad \text{and} \quad \vphi_1 + 2\vphi_2 \ge -\pi/2.$$
Claim (a) of Theorem \ref{t0} is proved.
\vspace{1mm}

(b) $\theta \in (0,\, \pi/4)$. Additionally to the inequalities in (a), we use that the slope of the surface of an optimal body at any regular point is either equal to 0, or greater than or equal to $\pi/4$ (Theorem 2.3 in \cite{BFK}). It follows that the angle between $e_i$,\, $i = 1,\,2$ and the $z$-axis is either 0, or $\ge \pi/4$. Since $\theta \ne 0$, these angles cannot be equal to zero. It follows that
$$\cos\vphi_1 \cos\theta \le 1/\sqrt 2 \quad \text{and} \quad \cos\vphi_2 \cos\theta \le 1/\sqrt 2,$$
hence
$$
|\vphi_1| \ge \vphi_*, \quad |\vphi_2| \ge \vphi_*, \quad \text{where} \ \ \vphi_* = \arccos\Big( \frac{1/\sqrt 2}{\cos\theta} \Big) < \frac{\pi}{4}.
$$
The admissible set is shown light gray in Fig.~\ref{fig2Dab}\,(a). It is the disjoint union of a quadrangle and two triangles, if $\theta \ge \arccos\sqrt{2/3}$, and a quadrangle, if $0 <\theta < \arccos\sqrt{2/3}$. Claim (b) of Theorem \ref{t0} is proved.
\vspace{1mm}

(c) $\theta = 0$. Each angle between $e_i$,\, $i = 1,\, 2$ and the $z$-axis is either equal to 0, or greater than or igual to $\pi/4$. It follows that  either $\vphi_i = 0$ or $|\vphi_i| \ge \pi/4$, $i = 1,\, 2$. Taking into account the inequalities of case (a), one concludes that $(\vphi_1, \vphi_2)$ either coincides with one of the points $(\pi/4, 0)$ and $(0, -\pi/4)$, or lies in the domain $2\vphi_1 + \vphi_2 \le \pi/2,\ \vphi_1 + 2\vphi_2 \ge -\pi/2,\ \vphi_1 \ge \pi/4,\ \vphi_2 \le -\pi/4$; see Fig.~\ref{fig2Dab}\,(b). Claim (c) of Theorem \ref{t0} is proved.

                                       \begin{figure}[h]
\centering
%\hspace*{6mm}
\includegraphics[scale=0.25]{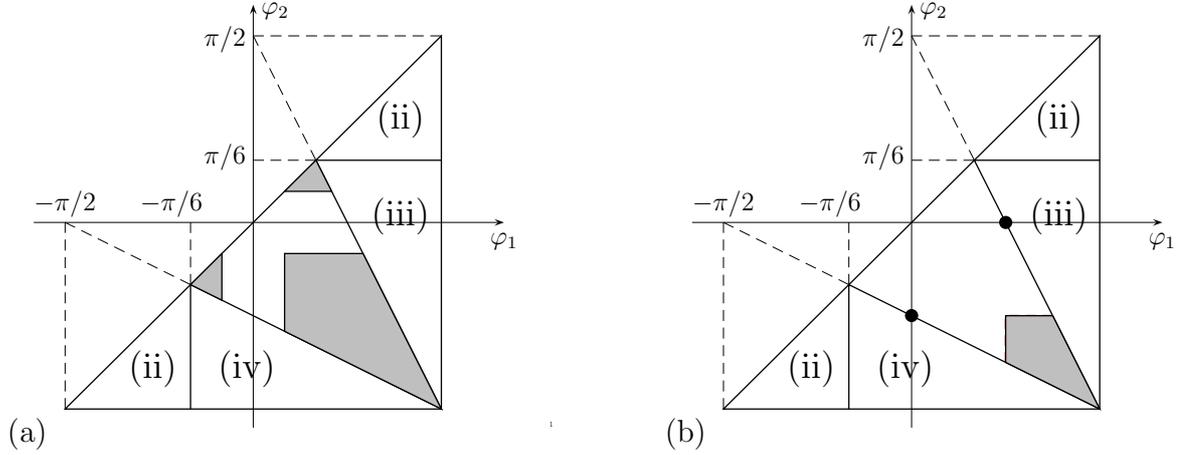}
\caption{(a) $\theta \ne 0$; (b) $\theta = 0$.}
\label{fig2Dab}
\end{figure}

\section*{Acknowledgements}

This work was supported by the Center for Research and Development in Mathematics and Applications (CIDMA) through the Portuguese Foundation for Science and Technology (FCT), within projects UIDB/04106/2020 and UIDP/04106/2020.

\end{document}